\documentclass[12pt,reqno,a4paper]{amsart}
\usepackage{latexsym,amsmath,amsfonts,amssymb,amsthm}
\theoremstyle{plain}
\newtheorem{Thm}{Theorem}
\newtheorem{Lem}{Lemma}

\theoremstyle{definition}
\newtheorem*{Ack}{Acknowledgment}
\theoremstyle{remark}

\def\Z{\mathbb Z}

\def\C{\mathbb C}

\def\T{\mathbb T}
\def\R{\mathbb R}
\def\P{\mathcal P}
\def\M{\mathfrak M}
\def\m{\mathfrak m}
\def\A{{\mathfrak A}}
\def\a{{\mathfrak a}}
\def\b{{\mathfrak b}}

\def\RR{\mathcal R}
\def\B{\mathcal B}

\def\1{{\bf 1}}
\def\pmod #1{\ ({\rm mod}\ #1)}

\def\floor #1{{\lfloor{#1}\rfloor}}

\begin{document}

\title{Two addition theorems on polynomials of prime variables}
\author{Hongze Li}
\email{lihz@sjtu.edu.cn}
\author{Hao Pan}
\email{haopan79@yahoo.com.cn}
\address{
Department of Mathematics, Shanghai Jiaotong University, Shanghai
200240, People's Republic of China} 
\thanks{This work was supported by
the National Natural Science Foundation of China (Grant No.
10471090).}
\maketitle

\section{Introduction}
\setcounter{Lem}{0}\setcounter{Thm}{0}\setcounter{Cor}{0}
\setcounter{equation}{0}

Recently, Khalfalah and Szemer\'edi \cite{KhalfalahSzemeredi06} proved the following theorem, which was conjectured by
Erd\H os, Roth, S\'ark\"ozy and S\'os \cite{ErdosSarkozy77}:
\begin{Thm}
Let $\psi$ be a polynomial with integral coefficients and positive
leading coefficient. Suppose that $\psi(1)\psi(0)$ is even. Then
for any $m$-coloring of all positive integers (i.e., partitioning
$\Z^+$ into $m$ disjoint non-empty subsets), there exist
monochromatic distinct $x,y$ such that $x+y=\psi(z)$ for an
integer $z$.
\end{Thm}
In particular, if all positive integers are colored with $m$-colors, then there exists
a monochromatic pair $x,y$ with $x\not=y$ such that $x+y$ is a perfect square.

On the other hand,
suppose that $\psi$ is a polynomial with rational coefficients and zero constant term,
in \cite{LiPan2} Li and Pan proved that for any subset $A$ of positive integers with
$$
\limsup_{x\to\infty}\frac{|A\cap[1,x]|}{x}>0,
$$
there exist $x,y\in A$ and a prime $p$ such that $x-y=\psi(p-1)$. This commonly generalizes two
well-known results of Furstenberg \cite{Furstenberg1977} and S\'ark\"ozy \cite{Sarkozy78a, Sarkozy78b}.

Define
$$
\lambda_{b,W}(x)=\begin{cases}
\frac{\phi(W)}{W}\log(Wx+b)&\qquad\text{if }Wx+b\text{ is prime},\\
0&\qquad\text{otherwise},
\end{cases}
$$
where $\phi$ is the Euler totient function and
$$
\Lambda_{b,W}=\{x:\,Wx+b\text{ is prime}\}
$$
for $1\leq b\leq W$ with $(b,W)=1$. In the present paper, our main result is the following theorem:
\begin{Thm}
\label{colorprimepoly} Let $m, b_0, W_0$ be positive integers
satisfying $b_0\leq W_0$ and $(b_0,W_0)=1$. Let $\psi(x)$ be a
polynomial with integral coefficients and positive leading
coefficient satisfying that
$$
\begin{cases}
\psi(1)\text{ or }\psi(0)\text{ is even}&\quad\text{ if }2\mid W_0,\\
\psi(b_0-1)\text{ is even}&\quad\text{ if }2\nmid W_0.
\end{cases}
$$
Suppose that all positive
integers are colored with $m$ colors. Then there exist
distinct monochromatic $x,y$ such that $x+y=\psi(z)$ where $z\in\Lambda_{b_0,W_0}$.
\end{Thm}
We shall use one of Green's ingredients in his proof of Roth's theorem in primes. The key of Green's
proof is a transference principle (which was greatly developed in \cite{GreenTao}), i.e.,
transferring a subset of primes with positive relative density to a subset
of $\Z_N=\Z/N\Z$ with positive density, where $N$ is a large prime. In the proof of Theorem \ref{colorprimepoly}, we shall
transfer one subset of $\{\psi(z):\,z\in\Lambda_{b,W}\}$
to a subset of $\Z_N=\Z/N\Z$ with the density very close to $1$.

\begin{Thm}
\label{primecolorprimepoly} Let $m, b_0, W_0$ be positive integers
satisfying $b_0\leq W_0$ and $(b_0,W_0)=1$. Let $\psi(x)$ be a
polynomial with integral coefficients and positive leading
coefficient satisfying that
$$
\begin{cases}
\psi(1)\text{ or }\psi(0)\text{ is even}&\quad\text{ if }2\mid W_0,\\
\psi(b_0-1)\text{ is even}&\quad\text{ if }2\nmid W_0.
\end{cases}
$$
Also, suppose that for each prime $p$, there exists $1\leq c_p\leq
p$ such that both $W_0c_p+b_0$ and $\frac{1}{2}\psi(c_p)$ are not
divisible by $p$. Then for any $m$-coloring of all primes, there
exist distinct monochromatic primes $x,y$ such that $x+y=\psi(z)$
where $z\in\Lambda_{b_0,W_0}$.
\end{Thm}

Let us explain why the existence of $c_p$ is necessary. Assume
that there exists a prime $p$ such that $c_p$ doesn't exist. That
is, for each $1\leq c\leq p$, either $W_0c+b_0$ or
$\frac{1}{2}\psi(c)$ is divisible by $p$. Then we may partition
the set of all primes into $3p$ disjoint sets $X_1,\ldots,X_{3p}$
with
$$
X_{j}= \{x\text{ is prime}:\,x\leq \psi((p-b_0)/W_0)/2, x\equiv
j\pmod{p}\},
$$
$$
X_{p+j}= \{x\text{ is prime}:\,x>\psi((p-b_0)/W_0), x\equiv
j\pmod{p}\}
$$
and
$$
X_{2p+j}= \{x\text{ is prime}:\,\psi((p-b_0)/W_0)/2<x\leq
\psi((p-b_0)/W_0), x\equiv j\pmod{p}\}.
$$
for $j=1,2,\ldots,p$. We claim that for each $1\leq j\leq 3p$, the
set
$$
\{(x,y,z):\, x,y\in X_j, z\in\Lambda_{b_0,W_0}, x\not=y,
x+y=\psi(z)\}
$$
is empty.

In fact, notice that now $p$ divides one of $W_0z+b_0$ and
$\frac{1}{2}\psi(z)$ since $c_p$ doesn't exist. If $p$ divides
$W_0z+b_0$, we must have $W_0z+b_0=p$ since $z\in\Lambda_{b_0,W_0}$.
But it is easy to see that for $1\leq j\leq p$
$$
\max\{x+y:\, x,y\in X_j,
x\not=y\}<2\cdot\psi((p-b_0)/W_0)/2=\psi(z),
$$
and for $p+1\leq j\leq 3p$
$$
\min\{x+y:\, x,y\in X_j,
x\not=y\}>2\cdot\psi((p-b_0)/W_0)/2=\psi(z).
$$
So it is impossible that
$$
\psi(z)\in X_j\dotplus X_j:=\{x+y:\, x,y\in X_j, x\not=y\}
$$
for any $1\leq j\leq 3p$.

On the other hand, suppose that $p$ divides $\frac{1}{2}\psi(z)$.
Note that for any $1\leq j\leq 3p$ and $x,y\in X_j$, $x\equiv
y\equiv j\pmod{p}$. So if $x+y=\psi(z)$, then we must have $x\equiv
y\equiv 0\pmod{p}$. Thus we have $x=y=p$ since  $x, y$ are both
primes. This also concludes that $\psi(z)\not\in X_j\dotplus X_j$
for each $j$.
\section{Proof of Theorem \ref{colorprimepoly}}

Assume that $n$ is a sufficiently large integer, and
$$
\{1,2,\ldots,n\}=X_1\cup\cdots\cup X_m
$$
where $X_i\cap X_j=\emptyset$ if $i\not=j$.
\begin{Lem} Let $p$ be a prime.
Let $h(x)$ be a non-zero polynomial over $\Z_p$.
Suppose that $S\subseteq\Z_p$ and $|S|\geq\deg h+1$. Then there exists $b\in S$ such that $h(b)\not\equiv 0\pmod{p}$.
\end{Lem}
\begin{proof}
This lemma easily follows from the fact that
$$
|\{x\in\Z_p:\,h(x)=0\}|\leq\deg h,
$$
since $h(x)$ doesn't vanish over $\Z_p$.
\end{proof}
Suppose $\psi(x)=a_1x^k+\cdots+a_kx+a_0$ be a polynomial with
integral coefficients. Let
$\Psi=\max\{(k+1)W_0,|a_1|,\ldots,|a_{k}|\}$. Let $\psi'$ denote
the derivative of $\psi$. Then for any prime $p>\Psi$, by Lemma
2.1, there exists $1\leq b_p\leq p-1$ with $b_p\equiv
b_0\pmod{W_0}$ such that $(\psi'((b_p-b_0)/W_0),p)=1$. And for
each prime $p\leq\Psi$, we may choose $b_p\geq 1$ with $p\nmid
b_p$ such that $b_p\equiv b_0\pmod{W_0}$ and
$\psi'((b_p-b_0)/W_0)>0$. In particular, we may assume that
$\psi((b_2-b_0)/W_0)$ is even if $2\mid W_0$. Let
$$
K=\prod_{\substack{p\text{ prime}\\
p\leq\Psi}}p^{\nu_p(\psi'((b_p-b_0)/W_0))},
$$
where $\nu_p(x)=\max\{v\in\Z:\,p^v\mid x\}$.

Let $\kappa=10^{-4}K^{-1}m^{-1}$.
Let $w=\floor{\log\log\log\log
n}$ and
$$
W=\prod_{\substack{p\text{ prime}\\ p\leq w}}p^w.
$$
Without loss of generality, we may assume that $w\geq\Psi$.
Suppose that $N$ is a prime in the interval
$(2n/W,(2+\kappa)n/W)]$. Thanks to the prime number theorem, such
prime $N$ always exists whenever $n$ is sufficiently large. By the
Chinese remainder theorem, there exists $0\leq b\leq W-1$ such
that for each prime $p\leq w$
$$
W_0b+b_0\equiv b_p\pmod{p^{w+\nu_p(W_0)}},
$$
since $b_p\equiv b_0\pmod{W_0}$. Clearly $(W_0b+b_0,WW_0)=1$. We
claim that $\psi(b)$ is even. In fact, when $W_0$ is odd, $b\equiv
b_2-b_0\equiv b_0-1\pmod{2}$. And if $W_0$ is even, we also have
$2\mid\psi(b)$ since $b\equiv (b_2-b_0)/W_0\pmod{2}$.

Define
$$
\psi_{b,W}(x)=\frac{\psi(Wx+b)-\psi(b)}{W}.
$$
Let $M=\max\{x\in \mathbb{N}:\psi_{b,W}(x)<KN\}$. Let $B$ be a
sufficiently large positive constant (only depending on $k$). Let
$$
\M_{a,q}=\{\alpha\in\T:\,|\alpha q-a|\leq (\log
M)^B/\psi_{b,W}(M)\},
$$
$$
\M=\bigcup_{\substack{1\leq a\leq q\leq (\log
M)^B\\(a,q)=1}}\M_{a,q}
$$
and $\m=\T\setminus\M$, where $\T=\R/\Z$.

\begin{Lem}
\label{primepolymajor}
For $\alpha\in\M_{a,q}$,
\begin{align*}
&\sum_{x=1}^M\psi_{b,W}^\Delta(x-1)\lambda_{W_0b+b_0,WW_0}(x)e(\alpha\psi_{b,W}(x))\\
=&\frac{\phi(WW_0)}{\phi(WW_0q)}\sum_{\substack{1\leq r\leq
q\\(WW_0r+W_0b+b_0,q)=1}}e(a\psi_{b,W}(r)/q)\sum_{x=1}^{\psi_{b,W}(M)}e((\alpha-a/q)\psi_{b,W}(x))\\
&+O(\psi_{b,W}(M)(\log M)^{-B}),
\end{align*}
where $\psi_{b,W}^\Delta(x)=\psi_{b,W}(x+1)-\psi_{b,W}(x)$.
\end{Lem}
\begin{Lem}
\label{primepolyminor} Suppose that $U\geq e^{a_1W^k}$. For any
$A>0$, there is a $B=B(A,k)>0$ such that,
$$
\sum_{x=1}^N\lambda_{b,W}(x)e(\alpha\psi(x))\ll_B N(\log N)^{-A}
$$
provided that $|\alpha-a/q|\leq q^{-2}$ with $1\leq a\leq q$,
$(a,q)=1$ and $(\log N)^B\leq q\leq\psi(N)(\log N)^{-B}$.
\end{Lem}
Lemma \ref{primepolymajor} is the immediate consequence of Lemmas
2.3 and 2.4 of \cite{LiPan2}. The proof of Lemma
\ref{primepolyminor} is standard but too long, so we omit the
details here. And the readers may refer to \cite{LiPan2} for the
proof.

Clearly $\psi_{b,W}$ is positive and strictly increasing on $[1,M]$ provided that $W$ is sufficiently large. Define
$$
\a(x)=\begin{cases} \psi_{b,W}^\Delta(z-1)\lambda_{W_0b+b_0,WW_0}(z)/\psi_{b,W}(M)&\text{if
}x=\psi_{b,W}(z)\text{ for a }1\leq z\leq M,\\
0&\text{otherwise}.
\end{cases}
$$
For any $f:\,\Z_N\to\C$, define
$$
\tilde{f}(r)=\sum_{x=1}^Nf(x)e(-xr/N).
$$
\begin{Lem}
\label{tildebound}
For any $0\not=r\in\Z_N$,
\begin{equation}
|\tilde{\a}(r)|\leq C_1Kw^{-\frac{1}{k(k+3)}},
\end{equation}
where $C_1$ is a constant (only depending on $k$).
\end{Lem}
\begin{proof}
If $r/N\in\m$, then by Lemma \ref{primepolyminor} and partial summation,
\begin{align*}
\tilde{\a}(r)=\frac{1}{\psi_{b,W}(M)}\sum_{z=1}^M\psi_{b,W}^\Delta(z-1)\lambda_{W_0b+b_0,WW_0}(z)
e(-\psi_{b,W}(z)r/N)\ll(\log M)^{-1}.
\end{align*}
Suppose that $r/N\in\M_{a,q}$. Then by Lemma \ref{primepolymajor}
\begin{align*}
&\frac{1}{\psi_{b,W}(M)}\sum_{z=1}^M\psi_{b,W}^\Delta(z-1)\lambda_{W_0b+b_0,WW_0}(z)e(-\psi_{b,W}(z)r/N)\\
=&\frac{\phi(WW_0)}{\phi(WW_0q)\psi_{b,W}(M)}\sum_{\substack{1\leq
s\leq q\\(WW_0s+W_0b+b_0,q)=1}}e(-\psi_{b,W}(s)a/q)\sum_{x=1
}^{\psi_{b,W}(M)}e(x(r/N-a/q))\\
&+O((\log M)^{-B})
\end{align*}
Notice that the leading coefficient of $\psi_{b,W}(x)$ is
$a_1W^{k-1}$, and the coefficient of $x^1$ in $\psi_{b,W}(x)$
coincides with
\begin{align*}
\psi_{b,W}'(0)=\frac{d}{dx}\bigg(\frac{\psi(Wx+b)-\psi(b)}{W}\bigg)\bigg|_{x\to
0}
=\frac{d\psi(x)}{dx}\bigg|_{x\to b}
=\psi'(b).
\end{align*}
Also, clearly for each prime
$p\leq w$, $\psi'(b)\equiv\psi'((b_p-b_0)/W_0)\pmod{p^w}$ since
$W_0b+b_0\equiv b_p\pmod{p^{w+\nu_{p}(W_0)}}$. Therefore when $w$ is sufficiently large, we have
$$
(\psi'(b), a_1W^{k-1})=(\psi'(b), W)=\prod_{p\leq
\Psi}p^{\nu_p(\psi'((b_p-b_0)/W))}=K.
$$
Thus by Lemma 2.7 of \cite{LiPan2},
$$
\sum_{\substack{1\leq s\leq q\\(WW_0s+W_0b+b_0,q)=1}}e(\psi_{b,W}(s)a/q)\ll
Kq^{1-\frac{1}{k(k+2)}}.
$$
Let $q_2$ be the largest divisor of $q$ prime to $W$ and
$q_1=q/q_2$. If $q\nmid W$, then either $q_2>w$ or $q\geq 2^w$.
Hence
\begin{align*}
&\frac{\phi(WW_0)}{\phi(WW_0q)\psi_{b,W}(M)}\sum_{\substack{1\leq s\leq
q\\(WW_0s+W_0b+b_0,q)=1}}e(\psi_{b,W}(s)a/q)\sum_{x=1
}^{\psi_{b,W}(M)}e(x(r/N-a/q))\\
\ll&\frac{Kq^{1-\frac{1}{k(k+2)}}}{q_1\phi(q_2)\psi_{b,W}(M)}
\bigg|\sum_{x=1
}^{\psi_{b,W}(M)}e(x(r/N-a/q))\bigg|\\
\ll& Kw^{-\frac{1}{k(k+3)}}.
\end{align*}

Below assume that $q\mid W$. Since $W$ divides the coefficients of
$x^i$ in $\psi_{b,W}(x)$ for $2\leq i\leq k$, we have
$$
\sum_{\substack{1\leq s\leq
q\\(WW_0s+W_0b+b_0,q)=1}}e(\psi_{b,W}(s)a/q)=\sum_{\substack{1\leq s\leq
q}}e(\psi'(b)sa/q)=\begin{cases} q&\text{if }q\mid (\psi'(b),W)=K,\\
0&\text{otherwise}.
\end{cases}
$$
Now suppose that $q\mid K$. Since
$KN-\psi_{b,W}(M)\leq\psi_{b,W}^\Delta(M)$, then
\begin{align*}
\sum_{x=1 }^{\psi_{b,W}(M)}e(x(r/N-a/q))=&\sum_{x=1
}^{KN}e(x(r/N-a/q))+O(\psi_{b,W}^\Delta(M))\\=&O(\psi_{b,W}^\Delta(M)).
\end{align*}
This concludes that if $q\mid W$ then
\begin{align*}
&\frac{\phi(WW_0)}{\phi(WW_0q)\psi_{b,W}(M)}\sum_{\substack{1\leq s\leq
q\\(WW_0s+W_0b+b_0,q)=1}}e(\psi_{b,W}(s)a/q)\sum_{x=1
}^{\psi_{b,W}(M)}e(x(r/N-a/q))\\
=&O((\log M)^{-B}).
\end{align*}
\end{proof}

By the pigeonhole principle, there exists $1\leq i\leq m$ such
that
$$
|\{x\in X_i\cap[\psi(W),n]:\, x\equiv \psi(b)/2\pmod{KW}\}|\geq
\frac{n}{mKW}-\psi(W)\geq\frac{N}{4mK}.
$$
Without loss of generality, we may assume that $X_1$ is such a
set. Let
$$
A=\{(x-\psi(b)/2)/W:\,x\in X_1\cap[\psi(W),n]:\, x\equiv
\psi(b)/2\pmod{KW}\}.
$$
Suppose that there exist $x',y'\in A$ and
$z'\in\Lambda_{W_0b+b_0,WW_0}$ such that $x'+y'=\psi_{b,W}(z')$.
Then letting $x=Wx'+\psi(b)/2$, $y=Wy'+\psi(b)/2\in X_1$ and
$z=Wz'+b\in\Lambda_{b_0,W_0}$, we have $x+y=\psi(z)$.

Below we consider $A$ as a subset of $\Z_N$. We claim that if
$x,y\in A$ and $z\in\Lambda_{W_0b+b_0,WW_0}\cap[1,M]$ satisfy
$x+y=\psi_{b,W}(z)$ in $\Z_N$, then the equality also holds in
$\Z$. Suppose that $x+y=\psi_{b,W}(z)-lN$ for an integer $l$. Then
$0\leq l<K$ since $n/W<N/2$ and $\psi_{b,W}(z)<KN$. Notice that
$K$ divides $x+y$ and all coefficients of $\psi_{b,W}$. We must
have $K\mid l$, whence $l=0$. Furthermore, we may consider $\a$ as a function over $\Z_N$, i.e.,
$$
\a(x)=\begin{cases} \frac{\psi_{b,W}^\Delta(z-1)}{\psi_{b,W}(M)}\lambda_{W_0b+b_0,WW_0}(z)&\text{if
}x=\psi_{b,W}(z)\text{ in }\Z_N\text{ for a }1\leq z\leq M,\\
0&\text{otherwise}.
\end{cases}
$$
This function is well-defined. In fact, assume that $1\leq z_1,z_2\leq M$ and $\psi_{b,W}(z_1)=\psi_{b,W}(z_2)$ in $\Z_N$. Then $\psi_{b,W}(z_1)=\psi_{b,W}(z_2)+lN$ in $\Z$
where $|l|<K$. But $\psi_{b,W}(z_1)\equiv\psi_{b,W}(z_2)\pmod{K}$, so $l=0$ and $z_1=z_2$.

Let $\eta$ and $\epsilon$ be two positive real numbers to be
chosen later. Let
$$
\RR=\{r\in\Z_N:\,|\tilde{\a}(r)|\geq\eta\}
$$
and
$$
\B=\{x\in\Z_N:\,\|xr/N\|\leq\epsilon\text{ for all }r\in\RR\},
$$
where $\|x\|=\min\{|x-z|:\,z\in\Z\}$. Define $\b=\1_\B/|\B|$ and
$\a'=\a*\b*\b$, where $\1_\B(x)=1$ or $0$  according to whether
$x\in\B$ or not and
$$
f*g(x)=\sum_{y\in\Z_N}f(y)g(x-y).
$$
\begin{Lem}
\label{aadashbound}
If $\epsilon^{|\RR|}\geq \kappa^{-1}C_1Kw^{-\frac{1}{k(k+3)}}$, then for any $x\in\Z_N$
$$
|\a'(x)|\leq\frac{1+2\kappa}{N}.
$$
\end{Lem}
\begin{proof}
It is easy to see that
$\widetilde{(f*g)}=\tilde{f}\cdot\tilde{g}$. By Lemma
\ref{primepolymajor} for $\alpha=0$ and Lemma \ref{tildebound},
\begin{align*}
|\a'(x)|=&\bigg|\frac{1}{N}\sum_{r}\tilde{\a}(r)\tilde{\b}(r)^2e(\frac{xr}{N})\bigg|\\
\leq&\frac{1}{N}|\tilde{\b}(0)|^2\sum_{z=1}^M\frac{\psi_{b,W}^\Delta(z-1)}{\psi_{b,W}(M)}\lambda_{W_0b+b_0,WW_0}(z)
+\frac{1}{N}\sup_{r\not=0}|\tilde{\a}(r)|\sum_{r\not=0}|\tilde{\b}(r)|^2\\
\leq&\frac{1+\kappa}{N}
+\frac{C_1Kw^{-\frac{1}{k(k+3)}}}{|\B|}.
\end{align*}
By the pigeonhole principle (cf. \cite[Lemma 1.4]{Tao}), we have
$|\B|\geq\epsilon^{|\RR|}N$. All are done.
\end{proof}
\begin{Lem}
\label{primepolyrestriction}
\begin{align*}
\sum_{r\in\Z_N}|\tilde{\a}(r)|^\rho\leq
C(\rho)K.
\end{align*}
provided that $\rho\geq k2^{k+3}$, where $C(\rho)$ is a constant
only depending on $\rho$.
\end{Lem}
\begin{proof}
Note that
\begin{align*}
\sum_{r\in\Z_N}|\tilde{\a}(r)|^\rho=\frac{1}{\psi_{b,W}(M)^\rho}\sum_{r\in\Z_N}\bigg|\sum_{z=1}^M\psi_{b,W}^\Delta(z-1)\lambda_{W_0b+b_0,WW_0}(z)e(-\psi_{b,W}(z)r/N)\bigg|^\rho.
\end{align*}
Thus Lemma \ref{primepolyrestriction} easily follows from Lemma 2.10 of \cite{LiPan2}.
\end{proof}
\begin{Lem}
\label{minor}
\begin{align*}
&\bigg|\sum_{\substack{1\leq x,y,z\leq N\\ x+y=z}}\1_A(x)\1_A(y)\a(z)-
\sum_{\substack{1\leq x,y,z\leq N\\ x+y=z}}\1_A(x)\1_A(y)\a'(z)\bigg|\\
\leq&C_2K(\epsilon^2\eta^{-k2^{k+3}}+\eta^{\frac{1}{k2^{k+3}+1}})N,
\end{align*}
where $C_2$ is a positive constant (only depending on $k$).
\end{Lem}
\begin{proof}
It is easy to see that
\begin{align*}
\sum_{\substack{1\leq x,y,z\leq N\\
x+y=z}}\1_A(x)\1_A(y)\a(z)=\frac{1}{N}\sum_{r\in\Z_N}\tilde{\1}_A(r)\tilde{\1}_A(-r)\tilde{\a}(r)
\end{align*}
and
\begin{align*}
\sum_{\substack{1\leq x,y,z\leq N\\
x+y=z}}\1_A(x)\1_A(y)\a'(z)=\frac{1}{N}\sum_{r\in\Z_N}\tilde{\1}_A(r)\tilde{\1}_A(-r)\tilde{\a}(r)\tilde{\b}(r)^2.
\end{align*}
Hence
\begin{align*}
&\sum_{\substack{1\leq x,y,z\leq N\\ x+y=z}}\1_A(x)\1_A(y)\a(z)-\sum_{\substack{1\leq x,y,z\leq N\\ x+y=z}}\1_A(x)\1_A(y)\a'(z)\\
=&\frac{1}{N}\sum_{r\in\Z_N}\tilde{\1}_A(r)\tilde{\1}_A(-r)\tilde{\a}(r)(1-\tilde{\b}(r)^2).
\end{align*}
Let $\rho=k2^{k+3}$. If $r\in\RR$, then by the proof of Lemma 6.7 of \cite{Green05}
$$
|1-\tilde{\b}(r)^2|\leq 32\epsilon^2.
$$
So
\begin{align*}
\bigg|\sum_{r\in\RR}\tilde{\1}_A(r)\tilde{\1}_A(-r)\tilde{\a}(r)(1-\tilde{\b}(r)^2)\bigg|\leq
|1-\tilde{\b}(r)^2|\sum_{r\in\RR}|\tilde{\1}_A(r)|^2|\tilde{\a}(r)|\leq64\epsilon^2N^2|\RR|.
\end{align*}
By Lemma \ref{primepolyrestriction} we have,
$$
|\RR|\leq\eta^{-\rho}\sum_{r\in\RR}|\tilde{\a}(r)|^\rho\leq C(\rho)K\eta^{-\rho}.
$$
Applying the H\"older inequality,
\begin{align*}
&\bigg|\sum_{r\not\in\RR}\tilde{\1}_A(r)\tilde{\1}_A(-r)\tilde{\a}(r)(1-\tilde{\b}(r)^2)\bigg|\\
\leq&\bigg|\sum_{r\not\in\RR}\tilde{\1}_A(r)\tilde{\1}_A(-r)\tilde{\a}(r)(1-\tilde{\b}(r)^2)\bigg|\\
\leq&
2N^\frac2{\rho+1}\sup_{r\not\in\RR}|\tilde{\a}(r)|^\frac1{\rho+1}\bigg(\sum_{r\not\in\RR}|\tilde{\1}_A(r)|^2\bigg)^{\frac{\rho}{\rho+1}}
\bigg(\sum_{r\not\in\RR}|\tilde{\a}(r)|^\rho\bigg)^{\frac1{\rho+1}}\\
\leq&2C(\rho)^{\frac1{\rho+1}}K^{\frac1{\rho+1}}\eta^\frac1{\rho+1}N^2,
\end{align*}
where we again use Lemma \ref{primepolyrestriction} in the last step.
\end{proof}

\begin{Lem}
\label{major}
\begin{align*}
\sum_{\substack{x,y,z\in\Z_N\\ x+y=z}}\1_A(x)\1_A(y)\a'(z)\geq\kappa^4 N.
\end{align*}
\end{Lem}
\begin{proof}
Let
$$
\A=\{x\in\Z_N:\,a'(x)\geq\kappa/N\}.
$$
Then by Lemma \ref{aadashbound}
$$
\frac{1+2\kappa}{N}|\A|+\frac{\kappa}{N}(N-|\A|)\geq\sum_{x\in\Z_N}\a'(x)=\sum_{x\in\Z_N}\a(x)\geq1-\kappa,
$$
whence $|\A|\geq (1-3\kappa)N$.
Define
$$
\nu_{A,A,-\A}(x)=|\{(x_1,x_2,x_3):\,x_1,x_2\in A, x_3\in\A, x_1+x_2-x_3=x\}|.
$$
By Lemma 3.3 of \cite{LiPan1}, we know
$$
\nu_{A,A,-\A}(x)\geq(\min\{|A|,|\A|,\frac{2|A|+|\A|-N}{4}\})^3N^{-1}.
$$
for any $x\in\Z_N$. It follows that
\begin{align*}
\sum_{\substack{x,y,z\in\Z_N\\ x+y=z}}\1_A(x)\1_A(y)\a'(z)
\geq\sum_{\substack{x,y\in A, z\in\A\\ x+y=z}}\frac{\kappa}{N}=
\frac{\kappa}{N}\nu_{A,A,-\A}(0)
\geq\kappa^4 N.
\end{align*}
\end{proof}
Combining Lemmas \ref{minor} and \ref{major}, we obtain that
\begin{align*}
&\sum_{\substack{1\leq x,y,z\leq N\\ x+y=z}}\1_A(x)\1_A(y)\a(z)\\
\geq&\sum_{\substack{1\leq x,y,z\leq N\\ x+y=z}}\1_A(x)\1_A(y)\a'(z)-
C_2K(\epsilon^2\eta^{-k2^{k+3}}+\eta^{\frac{1}{k2^{k+3}+1}})N\\
\geq&\kappa^4N-C_2K(\epsilon^2\eta^{-k2^{k+3}}+\eta^{\frac{1}{k2^{k+3}+1}})N.
\end{align*}
We may choose sufficiently small $\eta$ and $\epsilon$ such that $$\epsilon^{C(k2^{k+3})K\eta^{-k2^{k+3}}}\geq
\kappa^{-1}C_1Kw^{-\frac{1}{k(k+3)}}$$ and
$C_2K(\epsilon^2\eta^{-k2^{k+3}}+\eta^{\frac{1}{k2^{k+3}+1}})\leq \kappa^4/2$, provided that $w$ is sufficiently large. Thus
\begin{align*}
\sum_{\substack{x,y\in A,\ 1\leq z\leq N\\x\not=y,\
x+y=z}}\a(z)\geq&\sum_{\substack{x,y\in A,\ 1\leq z\leq N\\
x+y=z}}\a(z)-\sum_{\substack{1\leq z\leq
N}}\a(z)\geq\frac{\kappa^4}{3}N.
\end{align*}
All are done.\qed

\section{Proof of Theorem \ref{primecolorprimepoly}}
\setcounter{Lem}{0}\setcounter{Thm}{0}\setcounter{Cor}{0}
\setcounter{equation}{0}

Let $\P$ denote the set of all primes. Assume that
$\P=X_1\cup\cdots\cup X_m$ where $X_i\cap X_j=\emptyset$ if
$i\not=j$. Also, let $\kappa=10^{-4}K^{-1}m^{-1}$.

Let $\Psi=\max\{(2k+1)W_0,|a_1|,\ldots,|a_{k}|\}$. Then for a
prime $p>\Psi$, by Lemma 2.1 we know that there exists $1\leq
b_p\leq p-1$ with $b_p\equiv b_0\pmod{W_0}$ such that
$$
\psi'((b_p-b_0)/W_0)\psi((b_p-b_0)/W_0)\not\equiv0\pmod{p}.
$$
For a prime $p\leq\Psi$, we may choose $b_p\geq 1$ such that
$$
b_p\equiv W_0c_p+b_0\pmod{pW_0}
$$
and $\psi'((b_p-b_0)/W_0)>0$. Let
$$
K=\prod_{\substack{p\text{ prime}\\
p\leq\Psi}}p^{\nu_p(\psi'((b_p-b_0)/W_0))},
$$
where $\nu_p(x)=\max\{v\in\Z:\,p^v\mid x\}$.

Suppose that $n$ is a sufficiently large integer.
Let $w=\floor{\log\log\log\log n}$ and
$$
W=\prod_{\substack{\text{prime }\\
p\leq w}}p^w.
$$
Same as previous section, there exists $1\leq b\leq W-1$ such that
$$
W_0b+b_0\equiv b_p\pmod{p^{w+\nu_p(W_0)}}
$$
for each prime $p\leq w$. And also we know that $\psi(b)$ is even.

By the prime number theorem, we know
$$
\sum_{\substack{1\leq x\leq n,\ x\text{ prime}\\ x\equiv \psi(b)/2\pmod{KW}}}\log x=\frac{(1+o(1))n}{\phi(KW)}.
$$
Hence in view of the pigeonhole principle, without loss of generality, we may assume that
\begin{align*}
\sum_{\substack{x\in X_1\cap[\psi(W),n]\\ x\equiv \psi(b)/2\pmod{KW}}}\log x\geq&\frac{(1-\kappa)n}{m\phi(KW)}.
\end{align*}
Let $N$ be a prime in $(2n/W,(2+\kappa)n/W]$ and
$$
A=\{(x-\psi(b)/2)/W:\, x\in X_1\cap[\psi(W),n], x\equiv \psi(b)/2\pmod{KW}\}.
$$
Below we consider $A$ as a subset of $\Z_N$. Similarly, if
$x'+y'=\psi_{b,W}(z')$ holds in $\Z_N$ for $x',y'\in A$ and
$z'\in\Lambda_{W_0b+b_0,WW_0}$, then we also have $x+y=\psi(z)$
holds in $\Z$ where $x=Wx'+\psi(b)/2$, $y=Wy'+\psi(b)/2\in X_1$
and $z=Wz'+b\in\Lambda_{b_0,W_0}$. Define
$a=\1_A\lambda_{\psi(b)/2,KW}$/N. Clearly we have
\begin{align*}
\sum_{x=1}^Na(x)\geq&\frac{1}{3mK}.
\end{align*}
\begin{Lem}[Bourgain \cite{Bourgain89,Bourgain93} and Green \cite{Green05}]
\label{primerestriction}
$$
\sum_{r=1}^N|\tilde{a}(r)|^\rho\leq C'(\rho)
$$
for any $\rho>2$.
\end{Lem}
\begin{proof}
See \cite[Lemma 6.6]{Green05}.
\end{proof}
Let
$$
R=\{r\in\Z_N:\, |\tilde{a}(r)|\geq\eta\}$$
and
$$
B=\{x\in\Z_N:\, \|xr/N\|\leq\epsilon\text{ for all }r\in R\}.
$$
Define $\beta=\1_B/|B|$ and $a'=a*\beta*\beta$.
\begin{Lem}
\label{minorprime}
\begin{align*}
&\bigg|\sum_{\substack{1\leq x,y,z\leq N\\ x+y=z}}a(x)a(y)\a(z)-\sum_{\substack{1\leq x,y,z\leq N\\ x+y=z}}a'(x)a'(y)\a'(z)\bigg|\\
\leq&C_3K^\frac1{\rho+1}(\epsilon^2\eta^{-k2^{k+3}}+\eta^{\frac{1}{k2^{k+3}+1}})N^{-1},
\end{align*}
where $C_3$ is a positive constant (only depending on $k$).
\end{Lem}
\begin{proof}
We have
\begin{align*}
&\sum_{\substack{1\leq x,y,z\leq N\\ x+y=z}}a(x)a(y)\a(z)-\sum_{\substack{1\leq x,y,z\leq N\\ x+y=z}}a'(x)a'(y)\a'(z)\\
=&\frac{1}{N}\sum_{r\in\Z_N}\tilde{a}(r)\tilde{a}(-r)\tilde{\a}(r)\bigg(1-\tilde{\beta}(r)^2
\tilde{\beta}(-r)^2\tilde{\b}(r)^2\bigg).
\end{align*}
Let $\rho=k2^{k+3}$. If $r\in R\cap\RR$, then by Lemma 6.7 of
\cite{Green05},
$$
|1-\tilde{\beta}(r)^2) \tilde{\beta}(-r)^2\tilde{\b}(r)^2|\leq
2^{15}\epsilon^2.
$$
It follows that
\begin{align*}
&\bigg|\sum_{r\in
R\cap\RR}\tilde{a}(r)\tilde{a}(-r)\tilde{\a}(r)\bigg(1-\tilde{\beta}(r)^2
\tilde{\beta}(-r)^2\tilde{\b}(r)^2\bigg)\bigg|\\
\leq&
2^{15}\epsilon^2\sum_{r\in R\cap\RR}|\tilde{a}(r)|^2|\tilde{\a}(r)|\\
\leq&2^{16}\epsilon^2\min\{|R|, |\RR|\}.
\end{align*}
And by Lemma \ref{primepolyrestriction}, we have $|R|\leq C'(\rho)\eta^{-\rho}$.
Also, by the H\"older inequality, Lemmas \ref{primepolyrestriction} and \ref{primerestriction},
\begin{align*}
&\bigg|\sum_{r\not\in
R\cap\RR}\tilde{a}(r)\tilde{a}(-r)\tilde{\a}(r)\bigg(1-\tilde{\beta}(r)^2
\tilde{\beta}(-r)^2\tilde{\b}(r)^2\bigg)\bigg|\\
\leq& 2\sup_{r\not\in
R\cap\RR}|\tilde{a}(r)\tilde{\a}(r)|^\frac1{\rho+1}
\bigg(\sum_{r\not\in
R\cap\RR}|\tilde{a}(r)|^{2+\frac{1}{\rho}}\bigg)^{\frac{\rho}{\rho+1}}
\bigg(\sum_{r\not\in R\cap\RR}|\tilde{\a}(r)|^\rho\bigg)^{\frac1{\rho+1}}\\
\leq&2C'(2+1/\rho)^{\frac1{\rho+1}}C(\rho)^{\frac1{\rho+1}}K^\frac1{\rho+1}\eta^\frac1{\rho+1}.
\end{align*}
\end{proof}

\begin{Lem}
\label{adashbound} If $\epsilon^{|R|}\geq 2\log\log w/w$, then
$|a'(x)|\leq2/N$ for any $x\in\Z_N$.
\end{Lem}
\begin{proof}
See \cite[Lemma 6.3]{Green05}.
\end{proof}
Let
$$
A'=\{x\in\Z_N:\,a'(x)\geq\kappa/N\},\quad \A=\{x\in\Z_N:\,\a'(x)\geq\kappa/N\}.
$$
Then by the proof of Lemma \ref{major} we have $|\A|\geq
(1-3\kappa)N$. By Lemma \ref{adashbound} we have
$$
\frac{2}{N}|A'|+\frac{\kappa}{N}(N-|A'|)\geq\sum_{x\in\Z_N}a'(x)=\sum_{x\in\Z_N}a(x)\geq
\frac{1}{3mK},
$$
$$
|A'|\geq\frac{N}{2}\bigg(\sum_{x\in\Z_N}a'(x)-\frac{\kappa}{N}\cdot
N\bigg)=
\frac{N}{2}\bigg(\sum_{x\in\Z_N}a(x)-\frac{\kappa}{N}\cdot
N\bigg)\geq2\kappa N.
$$
Hence by Lemma 3.3 of \cite{LiPan1},
\begin{align*}
&\sum_{\substack{1\leq x,y,z\leq N\\ x+y=z}}a'(x)a'(y)\a'(z)\geq
\sum_{\substack{x,y\in A', z\in\A\\ x+y=z}}a'(x)a'(y)\a'(z)\geq \frac{\kappa^3}{N^3}\nu_{A',A',-\A}(0)\geq\frac{\kappa^6}{N}.
\end{align*}
We may choose sufficiently small $\eta$ and $\epsilon$ such that $$\epsilon^{C(k2^{k+3})K\eta^{-k2^{k+3}}}\geq
\kappa^{-1}C_1Kw^{-\frac{1}{k(k+3)}},$$
$$\epsilon^{C'(k2^{k+3})\eta^{-k2^{k+3}}}\geq
2\log\log w/w$$
and
$$
C_3K^{\frac{1}{\rho+1}}(\epsilon^2\eta^{-k2^{k+3}}+\eta^{\frac{1}{k2^{k+3}+1}})\leq
\kappa^6/2.
$$
So by Lemma \ref{minorprime}, we have
\begin{align*}
&\frac{\phi(KW)^2(\log(KWN+\psi(b)))^2}{K^2W^2N^2}\sum_{\substack{x,y\in A,\ 1\leq z\leq N\\ x\not=y,\ x+y=z}}\a(z)\\
\geq& \sum_{\substack{x,y,z\in\Z_N\\
x+y=z}}a'(x)a'(y)\a'(z)-C_3K^\frac1{\rho+1}(\epsilon^2\eta^{-k2^{k+3}}+\eta^{\frac{1}{k2^{k+3}+1}})N^{-1}\\
-&\frac{\phi(KW)^2(\log(KWN+\psi(b)))^2}{K^2W^2N^2}\sum_{\substack{1\leq z\leq N}}\a(z)\\
\geq&\kappa^6N^{-1}-C_3K^\frac1{\rho+1}(\epsilon^2\eta^{-k2^{k+3}}+\eta^{\frac{1}{k2^{k+3}+1}})N^{-1}-N^{-\frac{3}{2}}\\
\geq&\frac{\kappa^6}{3N}.
\end{align*}
\qed

\begin{Ack}
The second author thanks Professor Zhi-Wei Sun for informing the result of
Khalfalah and Szemer\'edi.
\end{Ack}

\end{document}